\documentclass[12pt,letterpaper]{article}
\usepackage{amsmath,amssymb,amsfonts,amsthm,mathrsfs, url,hyperref}
\usepackage{graphicx,enumerate}
\usepackage[usenames,dvipsnames]{color}

\usepackage[margin=1in]{geometry}

\usepackage[normalem]{ulem}

\def\dist{\text{dist}}

\theoremstyle{plain}
\newtheorem{thm}{Theorem}
\newtheorem{lem}[thm]{Lemma}
\newtheorem{cor}[thm]{Corollary}

\newtheorem{remark}[thm]{Remark}

\theoremstyle{definition}

\theoremstyle{remark}

\newcommand{\real}{\ensuremath {\mathbb R} }	

\newcommand{\nat}{\ensuremath {\mathbb N} }

\newcommand{\mbf}[1] {\text{\boldmath$#1$}}
\newcommand{\remove}[1] {}

\newcommand{\ex} {{\bf E}}
\newcommand{\pr} {{\bf Pr}}

\newcommand{\bfrac}[2]{\brac{\frac{#1}{#2}}}

\newcommand{\ra}{\rightarrow}
\newcommand{\la}{\leftarrow}
\newcommand{\multstar}[1]{\begin{multline*}#1\end{multline*}}

\newcommand{\Din}{\cD^{\text{in}}}
\newcommand{\Dout}{\cD^{\text{out}}}

\def\a{\alpha}

 \def\om{\omega}

\def\cG{{\cal G}}
\def\cD{{\mathcal D}}
\def\cA{\mathcal A}

\newcommand{\brac}[1]{\left(#1\right)}

\def\cH{{\mathcal H}}
\def\sP{{\mathscr P}}

\title{On the existence of Hamilton cycles with a periodic pattern in a random digraph}
\author{Alan Frieze\thanks{Department of Mathematical Sciences, Carnegie Mellon University, Pittsburgh PA, USA, 15213. Research supported in part by NSF grant DMS1661063.}, Xavier P\'erez-Gim\'enez\thanks{Department of Mathematics, University of Nebraska-Lincoln, Lincoln NE, USA, 68588. Research supported in part by Simons Foundation Grant \#587019.}, Pawe\l{} Pra\l{}at\thanks{Department of Mathematics, Ryerson University, Toronto, ON, Canada M5B 2K3. Research supported in part by NSERC Discovery Grant.}}
\date{}

\begin{document}
\maketitle
\begin{abstract}
We consider Hamilton cycles in the random digraph $\cD_{n,m}$ where the orientation of edges follows a pattern other than the trivial orientation in which the edges are oriented in the same direction as we traverse the cycle. We show that if the orientation forms a periodic pattern, other than the trivial pattern, then approximately half the usual $n\log n$ edges are needed to guarantee the existence of such Hamilton cycles a.a.s.
\end{abstract}

\section{Introduction}

The existence of Hamilton cycles is one of the key issues in the study of random graphs. The existence threshold for the random graph process $(\cG_{n,m})_{m=0}^{n (n-1) / 2}$ was found by Ajtai, Koml\'os and Szemer\'edi~\cite{KS} and the existence threshold for the directed analogue $(\cD_{n,m})_{m=0}^{n(n-1)}$ was found by Frieze~\cite{F1}. (See the next section for definitions of all models mentioned in the introduction as well as the asymptotic notation used.) There is a large literature on this subject and the reader is referred to the bibliography by Frieze \cite{F2} for more information.

The result in~\cite{F1} refers to Hamilton cycles in which the edges are oriented in the same direction round the cycle. Ferber and Long~\cite{FL} considered Hamilton cycles with prescribed edge orientations. They proved that if $C_n$ is an $n$-cycle with an arbitrary edge orientation then if $np\gg (\log\log n)\log n$ then $\cD_{n,p}$ contains a copy of $C_n$ a.a.s. (In fact, they proved more than this, in that they proved the existence of many copies.) They conjectured that, if $np=\log n+\om(1)$, then this is sufficient for $\cD_{n,p}$ to contain a copy of $C_n$ a.a.s. Montgomery \cite{M} has recently announced a proof of this conjecture. 

In this paper we consider Hamilton cycles with a periodic pattern of edge orientations and show that we need approximately half as many random edges to guarantee the existence of such Hamilton cycles a.a.s.

The paper is structured as follows. In the next section, we introduce all the definitions and state the main results. Section~\ref{idea} outlines the proof. Section~\ref{struct} provides structural results, and Section~\ref{cycles} uses these results to construct the claimed cycles. 

\section{Definitions and Main Results}

\subsection{Random Digraphs}

In this paper we present results obtained for the \emph{random directed graph} $\cD_{n,p}$. More precisely, $\cD_{n,p}$  is a distribution over the class of graphs with vertex set $[n]:=\{1,\ldots,n\}$ in which every ordered pair $ij$, $i,j \in [n]$ and $i \neq j$, appears independently as an arc in $\cD_{n,p}$ with probability~$p$. Note that $p=p(n)$ may (and usually does) tend to zero as $n$ tends to infinity. We say that $\cD_{n,p}$ has some property \emph{asymptotically almost surely} (or a.a.s.) if the probability that $\cD_{n,p}$ has this property tends to $1$ as $n$ goes to infinity.

As it is done in the theory of (undirected) random graphs, we slightly abuse the notation and for some natural number $m = m(n) \in \nat$ use $\cD_{n,m}$ to denote a directed graph on $n$ vertices and precisely $m$ arcs taken uniformly at random from the family of directed graphs on $n$ vertices and $m$ arcs.  Alternatively, $\cD_{n,m}$ can be constructed during the \emph{random digraph process}. Indeed, one may consider a sequence of digraphs $(\cD_{n,m})_{m=0}^{n(n-1)}$ with common vertex set $[n]$ in which $\cD_{n,0}$ is the empty digraph and $\cD_{n,m+1}$ is obtained from $\cD_{n,m}$ by adding a random arc that is not present in $\cD_{n,m}$. The process ends once $\cD_{n,n(n-1)}$, the complete digraph, is reached. 

\subsection{Asymptotic Notation and Convention}
 
Given two nonnegative functions $f=f(n)$ and $g=g(n)$, we will write $f=O(g)$ if there exists an absolute constant $c$ such that $f(n) \leq cg(n)$ for all $n \in \nat$, $f=\Omega(g)$ if $g=O(f)$, $f=\Theta(g)$ if $f=O(g)$ and $f=\Omega(g)$, and we write $f=o(g)$ or $f\ll g$ if the limit $\lim_{n\to\infty} f(n)/g(n)=0$. In addition, we write $f=\omega(g)$ or $f\gg g$ if $g=o(f)$. We also will write $f\sim g$ if $f=(1+o(1))g$. 

Through the paper, all logarithms with no subscript denoting the base will be taken to be natural. Moreover, as typical in the field of random graphs, for expressions that clearly have to be an integer, we round up or down but do not specify which, as long as this rounding does not affect the argument.

\subsection{Patterns}


Given a cycle $C=(v_1,\ldots, v_n,v_1)$ of length $n\ge1$ (where $C$ is a loop for $n=1$ and a double arc for $n=2$), we are interested in describing periodic orientations of the edges of $C$.
A~\emph{pattern} of length $k\ge1$ is a $k$-tuple $\pi = (\pi_1,\ldots,\pi_k)$ where $\pi_i\in\{\ra,\la\}$.
We say that an orientation of $C$ \emph{follows} pattern $\pi$ if $k\mid n$ and each edge $v_iv_{i+1}$ for $i\in[n]$ is oriented as
$\overrightarrow{v_iv_{i+1}}$ if $\pi_i=\ra$ or as $\overleftarrow{v_iv_{i+1}}$ otherwise.
(Here $v_{n+1}=v_1$, and the indices of $\pi_i$ should be taken modulo $k$.)
We identify a pattern $\pi$ with the only oriented cycle $C_\pi$ of length $k$ which follows $\pi$, and
consider two patterns $\pi$ and $\pi'$ to be \emph{equivalent} if their corresponding oriented cycles $C_\pi$ and $C_{\pi'}$ are isomorphic digraphs.
In other words, two patterns are equivalent if they can be obtained from one another by applying a cyclic rotation of the entries and/or a reflection (that is, reversing \emph{both} the order of the entries and the direction of the arrows).
A pattern is \emph{primitive} if it is not a concatenation of two or more consecutive copies of a shorter pattern.
We call an oriented cycle $C$ in a digraph $D$ a \emph{$\pi$-Hamilton cycle} (or \emph{$\pi$-HC} for short) if $C$ spans $D$ and follows pattern $\pi$.
Clearly, for two equivalent patterns $\pi$ and $\pi'$, a digraph $D$ contains a $\pi$-HC if and only if it contains a $\pi'$-HC.
Moreover, if $\pi$ is a primitive pattern of length $k$ and $\pi'$ is the concatenation of $\ell$ copies of $\pi$, then a digraph on $n$ vertices has a $\pi'$-HC if and only if it has a $\pi$-HC and $\ell k\mid n$.
In view of that, we can restrict our attention to the analysis of primitive patterns $\pi$.
All patterns of length $1$ are equivalent to $(\ra)$, which we call the \emph{trivial} pattern. Likewise, all primitive patterns of length $2$ are equivalent to $(\ra,\la)$. We refer to $(\ra,\la)$ as the \emph{alternating} pattern, so a $(\ra,\la)$-HC is simply called an alternating Hamilton cycle.
Primitive patterns of length $k\ge3$ are defined to be \emph{non-alternating}.

\begin{remark}\label{rem:nonprimitive}
While our main results concern only primitive patterns, in some parts of the argument it will be convenient to consider the non-primitive pattern $(\ra,\la,\ra,\la)$ as well.
This pattern will be useful when investigating the existence of alternating Hamilton cycles.
Indeed, when the number of vertices $n$ is divisible by $4$, a Hamilton cycle follows $(\ra,\la)$ if and only if it follows $(\ra,\la,\ra,\la)$.
In view of this, we can avoid some of the technical difficulties that arise in the analysis of patterns of length $2$ by simply redefining the alternating pattern to be $(\ra,\la,\ra,\la)$ when $4\mid n$.
The case when $n\equiv 2\mod 4$ can be reduced to the previous case with a slight modification of the argument.
\end{remark}

\remove{
OLD:
\bigskip

A \emph{pattern} of length $k$ is a cyclically ordered tuple $\pi = (\pi_1,\ldots,\pi_k)$ where $\pi_i\in\{\ra,\la\}$.
A~pattern is \emph{nontrivial} if it contains at least one edge in each direction. Since trivial patterns have already been studied, we will always assume that patterns are nontrivial. A pattern is \emph{primitive} if it is not a repetition of several consecutive copies of a smaller pattern. We will always assume that patterns are primitive. The trivial pattern $(\ra)$ (or $(\la)$) is the only pattern of order $1$, and the alternating pattern $(\ra,\la)$ (or $(\la,\ra)$) is the only pattern of order $2$.

An \emph{oriented cycle} $(v_1,\ldots, v_n)$ of length $n$ \emph{follows} pattern $\pi$ of length $k$ if $k\mid n$ and each edge $v_iv_{i+1}$ is oriented as $\pi_i$ with indices taken in arithmetic modulo $k$. We say that such cycle is a \emph{$\pi$-Hamilton cycle} (or \emph{$\pi$-HC} for short).
}

\subsection{Results}

We are now ready to state our main results.

\begin{thm}\label{thm1} Let $\pi=(\ra,\la)$ be the alternating pattern. Consider the random digraph process $(\cD_{n,m})_{m=0}^{n(n-1)}$ (restricted to even $n$). The following property holds a.a.s.:\ $\cD_{n,m}$ contains a $\pi$-HC the first time all vertices have in-degree at least 2 or out-degree at least 2.
\end{thm}

From this result, we immediately and easily establish the sharp threshold for the existence of the alternating Hamilton cycle, and obtain the limiting probability at the critical window. In particular, the following corollary holds.

\begin{cor}\label{cor1}
If $p = \frac{\log n + 2\log\log n + c}{2n}$ and $n$ is even, then
\[\pr(\cD_{n,p}\text{ contains an alternating Hamilton cycle}) \sim e^{-e^{-c}/4}.
\] 
\end{cor}

We obtain similar results for all primitive patterns of order $k\ge3$.

\begin{thm}\label{thm2} For any fixed $k\ge3$, let $\pi=(\pi_1,\ldots,\pi_k)$ be a primitive pattern. Consider the random digraph process $(\cD_{n,m})_{m=0}^{n(n-1)}$ (restricted to integers $n$ that are divisible by $k$). The following property holds a.a.s.:\ $\cD_{n,m}$ contains a $\pi$-HC the first time all vertices have total degree at least 2.
\end{thm}

\begin{cor}\label{cor2}
Let $\pi=(\pi_1,\ldots,\pi_k)$ be a primitive pattern of length $k\ge3$. If $p = \frac{\log n + \log\log n + c}{2n}$
and $k\mid n$, then
\[
\pr(\cD_{n,p}\text{ has $\pi$-HC}) \sim e^{-e^{-c}}.
\]
\end{cor}

\begin{remark}
Let us point out that there are different coefficients in the $\log\log n$ terms in Corollaries~\ref{cor1} and~\ref{cor2}. It follows from the fact that having a vertex with both in-degree and out-degree equal to 1 prevents there being a Hamilton cycle with the alternating pattern, but it does \emph{not} prevent any other pattern. See Lemma~\ref{lem:PH} to see how this affects the probability threshold. 
\end{remark}

\section{Idea of the Proof}\label{idea}

We start with a na\"\i ve description of the argument.
Let $k \ge 1$ be any fixed natural number, and let $\pi$ be any pattern of length $k$.
Suppose that we wish to find a $\pi$-Hamilton cycle in a given random digraph on $n$ vertices, where $n$ is divisible by $k$.
Here is one simple approach:
first partition the $n$ vertices into $k$ bins, $B_1,\ldots,B_k$, so that each bin receives exactly $n/k$ vertices (this is done arbitrarily, before examining the edges); then build a cycle $(v_1,\ldots,v_n,v_1)$, where each $v_i \in B_i$ and each $v_i v_{i+1}$ is oriented as $\pi_i$. (Here the indices of $B_i$ and $\pi_i$ should be interpreted modulo $k$). This gives us the desired $\pi$-HC.
Clearly, for the above construction to succeed, every vertex $v$ in bin $B_i$ should have some neighbours $w\in B_{i-1}$ and $w'\in B_{i+1}$ such that the edges $vw$ and $vw'$ are appropriately oriented.
Unfortunately, it is easy to see that for the $\cD_{n,p}$ model the above property fails a.a.s.\ when $k\ge2$ and $p\sim\log n/(2n)$ (which is the sharp threshold for the degree constraints given in Theorems~\ref{thm1} and~\ref{thm2}). 
This is due to the presence of vertices $v$ such that $v$ and all its neighbours belong to the same bin.
To overcome this obstacle, in this section we will analyze a simpler model of random digraphs for which the above na\"\i ve construction of a $\pi$-HC works.
Later in Section~\ref{cycles}, we will show how to modify $\cD_{n,p}$ (by moving a few vertices to different bins and contracting some short paths) so that the resulting random graph can be analyzed by means of this simpler model.

We consider a model that generalizes the $D_{s\text{-in}, t\text{-out}}$ model of~\cite{CF} as follows. Let $S=(s_{i,j})_{1\le i,j\le k}$ and $T=(t_{i,j})_{1\le i,j\le k}$ be two $k\times k$ matrices with entries in $\nat\cup\{0\}$. We construct the random digraph $D_{S\text{-in},T\text{-out}}$  as follows. We have $n$ vertices placed in $k\ge 1$ bins as explained above. For each vertex $v$ in bin $B_i$, we add $s_{i,j}$ distinct in-arcs to $v$ (that is, with head vertex $v$) and tail vertices chosen uniformly at random (u.a.r.) from $B_j \setminus \{v\}$. In other words, each of the $\binom{n/k}{s_{i,j}}$ choices (or $\binom{n/k-1}{s_{i,j}}$ choices if $i=j$) is equally likely. Similarly, for each vertex $v$ in bin $B_i$, we add $t_{i,j}$ distinct out-arcs from $v$ (that is, with tail vertex $v$) and head vertices chosen u.a.r.\ from $B_j \setminus \{v\}$. Let us note that $D_{S\text{-in}, T\text{-out}}$ is, in general, a multi-digraph since an arc $vw$ may be generated once as an out-arc from $v$ and once as an in-arc to $w$.

\medskip 
We start with the following observation.

\begin{lem}\label{lem1}
Let $S=T=\begin{pmatrix}0&2\\2&0\end{pmatrix}$. Then a.a.s.\ $D_{S\text{-in},T\text{-out}}$ has one perfect matching directed from $B_1$ to $B_2$ and one directed from $B_2$ to $B_1$.
\end{lem}
\begin{proof}
Walkup~\cite{W} (see Theorem 17.6 in~\cite{FK}) proved that the following model of a random bipartite graph has a perfect matching a.a.s: given disjoint sets $A,B$ of size $m$, each vertex of $A$ chooses 2 random neighbours in $B$ and each vertex of $B$ chooses 2 random neighbours in $A$. This is a random graph with $4m$ edges.  The random graph $D_{S\text{-in},T\text{-out}}$ consists of two independent copies of Walkup's graph, with orientations preserved, and so the conclusion follows immediately. 
\end{proof}

With one additional ingredient, namely, the fact that $D_{2\text{-in},2\text{-out}}$ a.a.s.\ has a directed Hamilton cycle, we get the following useful lemma. 
For technical reasons we disregard the case $k=2$, and restrict our attention to patterns of length $k\ge3$ (see Remark~\ref{rem:no2} below). Moreover, we do not assume patterns to be primitive so, in particular, the lemma applies to $\pi=(\ra,\la,\ra,\la)$.
\begin{lem}\label{mainlemma}
Let $k \ge 3$ be any natural number and let $\pi=(\pi_1,\ldots,\pi_k)$ be any pattern. 
Let $S=(s_{i,j})_{1\le i,j\le k}$ and $T=(t_{i,j})_{1\le i,j\le k}$ be $k\times k$ matrices with $s_{i,i+1}=s_{i+1,i}=t_{i,i+1}=t_{i+1,i}= 2$. Then a.a.s. $D_{S\text{-in},T\text{-out}}$ has a Hamilton cycle $(v_1,\ldots,v_n)$ with $v_i \in B_j$ when $i\equiv j\mod k$ and with each $v_i v_{i+1}$ oriented as $\pi_i$.
\end{lem}
\begin{proof}
Using Lemma~\ref{lem1}, we find a family of perfect matchings between $B_i$ and $B_{i+1}$ directed as $\pi_i$ for each $i=1,\ldots,k-1$. This gives a partition of the vertex set into sets that induce paths of length $k-1$. Each of these paths $P=(u_1,\ldots, u_k)$ has $u_i\in B_i$ and $u_iu_{i+1}$ oriented as $\pi_i$. We contract each path into a single vertex. Without loss of generality, we may assume that $\pi_k=\ \ra$; otherwise, a symmetric argument can be applied. It suffices to find a directed Hamilton cycle on the contracted paths, so that each edge is oriented from the $B_k$ end of one path to the $B_1$ end of the next one. In order to do this, given each path $P=(u_1, \ldots, u_k)$, note there are $t_{k,1}=2$ out-arcs from $u_k$ to a random vertex in $B_1$ and $s_{1,k}=2$ in-arcs to $u_1$ from a random vertex in $B_k$. Keeping only those arcs, we get a $D_{2\text{-in},2\text{-out}}$ graph on the vertices associated with the contracted paths. It was shown by Cooper and Frieze~\cite{CF2} that $D_{2\text{-in},2\text{-out}}$ a.a.s.\ has a directed Hamilton cycle. Un-contracting the paths gives us the desired $\pi$-Hamilton cycle of the original graph.
\end{proof}

\begin{remark}\label{rem:no2}
The proof of the lemma does not work `as is' for $k=2$ and $\pi=(\ra,\la)$, since the arcs between $B_1$ and $B_2$ used in the last step to complete a Hamilton cycle may have already been used in the first perfect matching between these same bins. With some extra work, one can show that the statement still holds in that scenario if we allow $S=T=\begin{pmatrix}0&4\\4&0\end{pmatrix}$.
We chose not to include this case in the statement, since it is not needed in our main argument.
\end{remark}

Lemma~\ref{mainlemma} will be useful in analyzing the random digraph process. However, the argument will be more delicate and the above proof strategy needs to be amended in order to deal with vertices of low degree that occur in the original process but not in the $D_{S\text{-in},T\text{-out}}$ one. These vertices will be ``buried'' inside paths after contracting them to ``fat'' vertices.

\section{Structural Ingredients}\label{struct}

\subsection{Bin partition}
Throughout this section, $\pi$ is a fixed primitive nontrivial pattern. We consider two different scenarios: either $\pi=(\ra,\la)$ (alternating case) or $\pi$ is a primitive pattern $(\pi_1,\ldots,\pi_k)$ of length $k\ge3$ (non-alternating case).
We will analyze several models of random digraphs with common vertex set $[n]$.
In the non-alternating case, we assume that the length of the pattern $k$ divides $n$, and partition the vertices into $k$ bins $B_1,\ldots,B_k$ in an equitable manner (that is,~every bin receives exactly $n/k$ vertices).
The alternating case is slightly different. Here we assume that $n$ is even, but partition the vertices into $4$ bins $B_1,\ldots,B_4$, so that bins $B_1,B_2$ receive $\lceil n/4\rceil$ vertices and bins $B_3,B_4$ receive $\lfloor n/4\rfloor$ vertices.
By setting $k=4$, the alternating case can be intuitively regarded as if $\pi=(\ra,\la,\ra,\la)$ instead of $(\ra,\la)$, with the weaker requirement that $n$ is even but not necessarily divisible by $4$ (see Remark~\ref{rem:nonprimitive}).
In either case, the partition of the vertex set into bins is done arbitrarily before examining the arcs of any of the random digraphs.

\subsection{Necessary conditions}

We now describe what will turn out to be the main obstruction to the existence of a $\pi$-HC in $\cD_{n,p}$.
In the non-alternating case (that is, $\pi$ primitive of length $k\ge3$), let $\cA$ denote the event that $\cD_{n,p}$ has no vertices of total degree less than $2$. In the alternating case, we require a stronger condition, and define $\cA$ to be the event that every vertex in $\cD_{n,p}$ has in-degree at least $2$ or out-degree at least $2$. Note that in either case $\cA$ is trivially a necessary condition for the existence of $\pi$-HC.

The following result is a standard exercise in the field of random graphs, so we just sketch the main steps of the proof. 
\begin{lem}\label{lem:PH}
Let $c\in\real$ be any fixed constant.
\begin{enumerate}
\item
For $\pi=(\ra,\la)$ and with $p_{\text{alt}} =p_{\text{alt}}(n) = \frac{\log n + 2\log\log n +c}{2n}$
\[
\pr(\cD_{n,p_{\text{alt}}} \text{ satisfies } \cA) \sim e^{-e^{-c}/4}.
\]
\item
For a primitive pattern $\pi$ of length $k\ge 3$ and with $p_{\text{non-alt}} = p_{\text{non-alt}}(n) = \frac{\log n + \log\log n + c}{2n}$, 
\[
\pr(\cD_{n,p_{\text{non-alt}}} \text{ satisfies } \cA) \sim e^{-e^{-c}}.
\]
\end{enumerate}
\end{lem}
\begin{proof}[Proof (sketch)]
Let us first consider the case $\pi=(\ra,\la)$. Let $X$ denote the number of vertices in $\cD_{n,p_{\text{alt}} }$ with in-degree $1$ and out-degree $1$. Easy computations show that
\[
\ex X \sim n^3 p_{\text{alt}}^2(1- p_{\text{alt}} )^{2n} \sim \frac {n \log^2 n}{4} \exp \Big( - (\log n + 2 \log \log n + c) \Big) = e^{-c}/4.
\]
Similarly, the $i$th factorial moment satisfies $\ex [ (X)_i ] \sim (e^{-c}/4)^i$ for each fixed $i\in\nat$. As a result, it follows that $X$ is asymptotically Poisson with mean $\lambda := e^{-c}/4$. In particular, $\pr(X=0) \sim e^{-\lambda}$. Moreover, the expected number of vertices of total degree at most $1$ is equal to $O(1/\log n)=o(1)$, and so
\[
\pr(\cA) = \pr(X=0) + o(1) \sim e^{-e^{-c}/4}.
\]

The non-alternating case is very similar. This time we need to investigate random variable $Y$, the number of vertices in $\cD_{n,p_{\text{non-alt}} }$ with total degree $1$. We get that
$$
\ex Y \sim n^2 (2p_{\text{non-alt}}) (1- p_{\text{non-alt}})^{2n} \sim n \log n \exp \Big( - (\log n + \log \log n + c) \Big) = e^{-c},
$$
and so
$$
\pr(\cA) = \pr(Y=0) + o(1) \sim e^{-e^{-c}}.
$$
The details are left for the reader.
\end{proof}

\subsection{In- and out-arcs}

For each $p=p(n)\in[0,1]$, we can view $\cD_{n,p}$ as the union of two independent copies of $\cD_{n,p'}$, denoted $\Din_{n,p'}$ and $\Dout_{n,p'}$, where $p'=p'(n)$ is chosen such that
\[
2p' - {p'}^2 = p.
\]
This is a standard, simple but useful, observation called \emph{two-round exposure} or \emph{sprinkling}. The arcs in $\Din_{n,p'}$ receive label``in'', and we will call them {\em in-arcs}. Likewise, the arcs in $\Dout_{n,p'}$ are labeled ``out'', and will be called {\em out-arcs}.
Note that an arc $\overrightarrow{vw}$ may appear in both $\Din_{n,p'}$ and $\Dout_{n,p'}$,
which creates a parallel pair of arcs, one in-arc $\overrightarrow{vw}$ and one out-arc $\overrightarrow{vw}$.
This is useful for the purpose of building a Hamilton cycle, since we may use either copy of $\overrightarrow{vw}$ for the cycle.
By forgetting the labels and merging parallel pairs into single arcs, we recover the usual distribution of $\cD_{n,p}$.
Hence, we regard $\cD_{n,p}$ as a simple digraph, where each arc has label ``in'', ``out'', or both.

Furthermore, we wish to build the standard random digraph process $(\cD_{n,p})_{0\le p\le 1}$ in a way that is compatible with the arc labels. To do so, we introduce a sequence $(X^{\text{in}}_{v,w}, X^{\text{out}}_{v,w})_{v,w\in[n], v\ne w}$ of i.i.d.\ random variables uniformly distributed in $[0,1]$. Then, for each pair of different vertices $v,w$, we include arc $\overrightarrow{vw}$ in $\Din_{n,p'}$ (or in $\Dout_{n,p'}$)  if $X^{\text{in}}_{v,w} \le p'$ (or, respectively, $X^{\text{out}}_{v,w} \le p'$).
Hence, setting $\cD_{n,p}=\Din_{n,p'}\cup\Dout_{n,p'}$ (and merging parallel pairs into single arcs), we obtain three random processes $(\cD_{n,p})_{0\le p\le 1}$, $(\Din_{n,p'})_{0\le p'\le 1}$ and $(\Dout_{n,p'})_{0\le p'\le 1}$ with the usual couplings: $\cD_{n,p_1}\subseteq \cD_{n,p_2}$, $\Din_{n,p'_1}\subseteq \Din_{n,p'_2}$ and $\Dout_{n,p'_1}\subseteq \Dout_{n,p'_2}$ for $0\le p_1\le p_2\le 1$.
Note that, if $X^{\text{in}}_{v,w} = p_1' < p_2'= X^{\text{out}}_{v,w}$, we say that arc $\overrightarrow{vw}$ appears in $(\cD_{n,p})_{0\le p\le 1}$ at time $p_1$ with label ``in'' and gets a second label ``out'' at time $p_2$. However, if we choose to ignore labels, then arc $\overrightarrow{vw}$ simple appears in $(\cD_{n,p})_{0\le p\le 1}$ at time $p_1$.

\medskip

In the context of the random digraph process $(\cD_{n,p})_{0\le p\le 1}$, we consider random variable
\[
p_* = \min \{ p\in[0,1] : \cD_{n,p} \text{ satisfies } \cA\}.
\]
Let $\om:=\log\log\log n = o(\log \log n)$ and let
\begin{equation}\label{eq:ppm}
p_{\pm} = \begin{cases}
\dfrac{\log n + 2\log\log n \pm\om}{2n}
& \text{if $\pi$ is the alternating pattern ($k=2$)},
\\\\
\dfrac{\log n + \log\log n \pm\om}{2n}
& \text{otherwise ($k\ge3$)}.
\end{cases}
\end{equation}
In particular, by Lemma~\ref{lem:PH}, a.a.s.\ $p_- < p_* < p_+$.

In the remainder of this section, we will analyze the two random digraph processes $(\Din_{n,p'})_{p'_-\le p'\le  p'_+}$ and $(\Dout_{n,p'})_{p'_-\le p'\le  p'_+}$, where $2p'_\pm - {p'_\pm}^2 = p_\pm$. These two processes, in turn, determine the process $(\cD_{n,p})_{p_-\le p\le  p_+}$, as explained above. For each ``time'' $p'\in [p'_-, p'_+]$, we associate out-arcs (that is, arcs in $\Dout_{n,p'}$) to their tail vertex and in-arcs to their head vertex. More precisely, if $\overrightarrow{vw}$ is an out-arc with tail at $v$ and head at $w$, then we say that it is {\em visible} from $v$ and {\em invisible} from $w$. Similarly, if $\overrightarrow{vw}$ is an in-arc, then it is {\em invisible} from $v$ and {\em visible} from $w$.
Later in the argument, we will expose the endpoint of each arc from which the arc is visible and leave the other endpoint random.
This trick will be useful to emulate the  $D_{S\text{-in},T\text{-out}}$ model.

\subsection{Useful properties}

Recall that vertices are equitably partitioned into bins $B_1,\ldots,B_k$ before exposing any random arcs.
(In the alternating case, $k=4$ and $|B_1|=|B_2|=\lceil n/4\rceil$ while $|B_3|=|B_4|=\lfloor n/4\rfloor$.)
We use this partition to classify vertices into three different types (good, bad, or dangerous): 
\begin{itemize}
\item
We say a vertex $v$ is {\em good} if, for every $i\in[k]$, $v$ has at least $4k+2$ visible out-arcs (that is,~out-arcs visible from $v$) in $\Dout_{n,p'_-}$ whose head vertex is in bin $B_i$ and also at least $4k+2$ visible in-arcs in $\Din_{n,p'_-}$ whose tail is in $B_i$. (Note that this property is required for each bin, including the bin vertex $v$ belongs to.) 
\item
We say that $v$ is {\em bad} if it is not good, but its total degree in $\cD_{n,p_-}=\Din_{n,p'_-}\cup\Dout_{n,p'_-}$ (including all arcs that are incident with $v$, visible and invisible, and ignoring labels) is at least $4k+3$.
\item
The remaining vertices, which have total degree at most $4k+2$, are called {\em dangerous}.
\end{itemize}
The above classification of vertices is based solely on the two digraphs $\Din_{n,p'_-}$ and $\Dout_{n,p'_-}$. It remains invariant throughout $(\Din_{n,p'})_{p'_-\le p'\le  p'_+}$, $(\Dout_{n,p'})_{p'_-\le p'\le  p'_+}$, and hence  $(\cD_{n,p})_{p_-\le p\le  p_+}$. In other words, if a vertex is good/bad/dangerous at time $p_-$ it stays of this type during any time $p \in[p_-,  p_+]$.

\medskip

Given a digraph $\cD$, the {\em $\cD$-distance} between two vertices of $\cD$ denotes the usual graph distance in the underlying undirected graph. 
Let $\cH$ (for handsome) be the family of all digraphs $\cD$ with $\cD_{n,p_-} \subseteq \cD \subseteq \cD_{n,p_+}$ satisfying:
\begin{enumerate}[{\bf H1:}]
\item
Every vertex in $\cD$ has fewer than $4k$ vertices that are not good within $\cD$-distance $10k$.
\item
Every dangerous vertex in $\cD$ has only good vertices within $\cD$-distance $10k$.
\item
Every cycle (with any orientation, any labels on the associated arcs, and including cycles of length $2$) in $\cD$ of length at most $10k$ has only good vertices within $\cD$-distance $10k$.
\end{enumerate}
Note that distances in {\bf H1}--{\bf H3} are measured with respect to digraph $\cD$, whereas the definitions of being good, bad, and dangerous depend only on arcs that are already present in $\cD_{n,p_-} \subseteq \cD \subseteq \cD_{n,p_+}$.
As a result, $\cH$ is a monotone decreasing family of digraphs. 

\medskip

We will show now that random directed graphs are typically handsome. 

\begin{lem}\label{lem:handsome}
A.a.s.\ $\cD_{n,p}\in \cH$ for all $p_-\leq p\leq p_+$.
\end{lem}
\begin{proof}
As mentioned above, {\bf H1}--{\bf H3} are monotone decreasing properties, so it suffices to show that a.a.s.\ $\cH$ holds for $p=p_+$.
We will investigate these three events in turn.
\begin{enumerate}[{\bf H1:}]
\item Let $M=40k^2$. Then, 
\begin{align*}
\pr(\exists & \text{ a vertex of $\cD_{n,p_+}$ not satisfying {\bf H1}}) \\
\leq & \ n\sum_{\ell=4k}^M\binom{n}{\ell}(\ell+1)^{\ell-1}(2p_+)^{\ell}\binom{\ell}{4k}\brac{2k\sum_{i=0}^{4k+1} \binom{\lceil n/k\rceil}{i}p_-^i(1-p_-)^{\lfloor n/k\rfloor-M-1-i}}^{4k} \\
\le & \ n\sum_{\ell=4k}^M\bfrac{ne}{\ell}^\ell\bfrac{2(\ell+1) \log n}{n}^\ell \\
& \qquad \times \brac{ \frac{e \ell}{4k} (2k) 2 (n/k)^{4k+1} (\log n / n)^{4k+1} \exp \left( - (1+o(1)) \frac {\log n}{2n} \cdot \frac {n}{k} \right) }^{4k} \\
\le & \ n\sum_{\ell=4k}^M \brac{10 \log n}^\ell \brac{\log^{4k+1}n\cdot n^{-1/(2k)+o(1)}}^{4k} = n^{1-2+o(1)} = o(1).
\end{align*}
Indeed, there are $n$ choices for a vertex $v$ and then there will be an additional $\ell$ vertices that make up the paths, where $4k \le \ell\leq M = (4k)\times(10k)$. We choose these vertices in at most $\binom{n}{\ell}$ ways and then choose a spanning tree that is contained in the union of the paths in at most $(\ell+1)^{\ell-1}$ ways, the number of labelled trees on $\ell+1$ vertices. The factor $(2p_+)^\ell$ bounds the probability the arcs of the tree exist, {\em ignoring orientation} gives us the 2. We then choose $4k$ vertices that are not good and multiply by an upper bound for the probability that these vertices have few $\cD_{n,p_-}$ in- or out-neighbours in some of the $k$ bins outside of the set of $\ell+1 \le M + 1$ vertices chosen so far. 
\item The probability that there is a dangerous vertex $v$ and a not good vertex $w$ such that $\dist(v,w)=\ell\leq 10k$ can be bounded from above by 
\begin{align*}
& n \sum_{\ell=1}^{10k}n^{\ell}(2p_+)^\ell \brac{2k\sum_{i=0}^{4k+1} \binom{\lceil n/k\rceil}{i}p_-^i(1-p_-)^{\lfloor n/k\rfloor-\ell-1-i}} \\ 
& \qquad\qquad\qquad \times \brac{\sum_{j=0}^{4k+1}\binom{n}{j}(2p_-)^j(1-2p_-+p_-^2)^{n-2-j}} \\
& \quad \leq n \sum_{\ell=1}^{10k} \Big( (1+o(1))\log n \Big)^\ell n^{-1/(2k)+o(1)} n^{-1+o(1)} = n^{-1/(2k) + o(1)} = o(1).
\end{align*}
Indeed, there are $n$ choices for a dangerous vertex $v$ and then at most $n^{\ell}$ choices for additional vertices that form a path of length $\ell$ reaching a vertex $w$ that is not good. The next term is an upper bound for the probability that $w$ is not good. The last term is an upper bound for the probability that $v$ is dangerous. Note that $v$ already has one neighbour on the path that is already identified; hence, $j \le 4k+1$, not $4k+2$.

\item 
If there is a cycle not satisfying~{\bf H3}, then there is a set of $2\leq \ell\leq M=10k+10k = 20k$ vertices containing at least $\ell$ arcs and at least one not good vertex. The probability of this event occurring in $\cD_{n,p}$ can be bounded from above by
\multstar{
\sum_{\ell=2}^M\binom{n}{\ell}\binom{\ell^2}{\ell}(2p_+)^\ell \ell (2k) \sum_{i=0}^{4k+1} \binom{\lceil n/k\rceil}{i}p_-^i(1-p_-)^{\lfloor n/k\rfloor-M-i} \\
\leq 40k^2\sum_{\ell=2}^M\bfrac{ne}{\ell}^\ell \bfrac{\ell^2e}{\ell}^\ell \bfrac{2 \log n}{n}^\ell n^{-1/(2k)+o(1)}=n^{-1/(2k)+o(1)}=o(1).
}
\end{enumerate}
The proof of the lemma is finished.
\end{proof}

Given a digraph $\cD_{n,p_-} \subseteq \cD \subseteq \cD_{n,p_+}$ and a vertex $v$ of $\cD$, the neighbourhood $N_{\cD}(v)=N(v)$ of $v$ is the set of all vertices $w$ that are joined to $v$ by an arc (in- or out-arc, visible or invisible from $v$). We now give some useful deterministic consequences of properties {\bf H1}--{\bf H3}.

\begin{lem}\label{lem:useful}
Let $\cD$ be any digraph in $\cH$, and let $\sP_0$ be a pairwise vertex-disjoint collection of paths, where each path $P\in\sP_0$ has length at most $6k+2$ and contains exactly one non-good vertex (that is, bad or dangerous). Then
\begin{enumerate}[{\bf H1}]
\addtocounter{enumi}{3}
\item
For every dangerous vertex $v$, all the vertices in $N(v)$ are good and none of them is contained in a path in $\sP_0$.
\item
For every vertex $v$, all but at most $4k$ vertices in $N(v)$ are good and are not contained in a path in $\sP_0$. 
\end{enumerate}
\end{lem}
\begin{proof}
Pick any vertex $v$. {\bf H3} implies that no path $P\in\sP_0$ intersects more than one vertex in $N(v)$, since otherwise we would create a cycle of length at most $10k$ with at least one non-good vertex within distance $10k$. Then, by {\bf H1}, we immediately get {\bf H5}. Moreover, if $v$ is dangerous, then {\bf H2} implies {\bf H4}.
\end{proof}

\section{Cycle Construction}\label{cycles}
In this section we will prove Theorems~\ref{thm1} and~\ref{thm2} by showing that a.a.s.\ $\cD_{n,p_*}$ contains a $\pi$-Hamilton cycle, where $\pi$ is either the alternating pattern $(\ra,\la)$ or a primitive pattern $(\pi_1,\ldots,\pi_k)$ of length $k\ge 3$.
For notational convenience, in the alternating case we set $k=4$ and redefine $\pi$ to be $(\ra,\la,\ra,\la)$, even though $n$ is only assumed to be even but not necessarily divisible by $4$ (see~Remark~\ref{rem:nonprimitive}). In the non-alternating case, we always require that $k\mid n$.
As in Section~\ref{struct}, the vertex set $[n]$ is equitably partitioned into $k\ge3$ bins $B_1,\ldots,B_k$ of size $|B_i|=n/k$ (in the non-alternating case) or $|B_1|=|B_2|=\lceil n/4\rceil$ and $|B_3|=|B_4|=\lfloor n/4\rfloor$ (in the alternating case).
Recall that this partition is done before exposing any random arcs.
\medskip

We will analyze the random processes $(\Din_{n,p'})_{p'_-\le p'\le p'_+}$, $(\Dout_{n,p'})_{p'_-\le p'\le p'_+}$ and $(\cD_{n,p})_{p_-\le p\le p_+}$ introduced in Section~\ref{struct}, where $\cD_{n,p} = \Din_{n,p'} \cup \Dout_{n,p'}$ for $2p'-{p'}^2=p$ (merging parallel arcs into single arcs) and where $p_-,p_+$ are defined as in~\eqref{eq:ppm}.
In view of Lemmas~\ref{lem:PH} and~\ref{lem:handsome}, we will assume that $p_- < p_* < p_+$ and 
that event $\cH$ holds in $\cD_{n,p}$ for all $p_-\le p \le p_+$. Otherwise our construction will simply fail, but this occurs with probability $o(1)$.
Note that we do not condition on these events, since this would destroy the probability space. Instead, we will expose some partial information about $(\cD_{n,p})_{p_-\le p\le p_+}$ and assume it does not contradict our assumptions above, while we leave all the unexposed information random.

\paragraph{Step 1 -- digraph $\mbf{\cD_*}$:}
We first consider digraphs $\Din_{n,p'_-}$ and $\Dout_{n,p'_-}$ (or equivalently digraph $\cD_{n,p_-}$ at time $p_-$). Recall that an in-arc $\overrightarrow{vw}$ of $\Din_{n,p'_-}$ is visible from $w$, while
an out-arc $\overrightarrow{vw}$ of $\Dout_{n,p'_-}$ is visible from $v$.
For every vertex $v$ and for every bin $B_i$, we expose the number $d^-_i(v)$ of in-arcs in $\Din_{n,p'_-}$ with head at $v$ and tail at some vertex in $B_i$ and similarly the number $d^+_i(v)$ of out-arcs in $\Dout_{n,p'_-}$ with tail at $v$ and head at some vertex in $B_i$. (In other words, we reveal the number of arcs between $v$ and $B_i$ in each orientation and which are visible from $v$.)
This allows us to identify which vertices are good without exposing actual locations of arcs. Indeed, given $d^-_i(v)$ and $d^+_i(v)$, the unexposed endvertices of the $d^+_i(v)$ out-arcs and $d^-_i(v)$ in-arcs visible from $v$ are two subsets of $B_i\setminus\{v\}$ of sizes $d^+_i(v)$ and $d^-_i(v)$ chosen independently and u.a.r.\ (and also independently of the arcs that are visible from other vertices).

For every vertex $v$ that is not good, we expose all arcs in $\cD_{n,p_-} = \Din_{n,p'_-} \cup \Dout_{n,p'_-}$ (in- and out-arcs, visible and invisible) that are incident with $v$, and label these arcs {\em discovered}.
This information determines whether $v$ is bad or dangerous.
Hence, all vertices are classified into the good, bad and dangerous types defined in Section~\ref{struct}.
We also label every vertex that is not good as {\em discovered}.
(Good vertices remain {\em undiscovered} for now, but this will change later on, as we reveal more information about $\cD_{n,p_-}$ and label some good vertices as discovered.)
For each $i\in[k]$, let $\hat B_i$ denote the set of undiscovered vertices in bin $B_i$. Note that some of the arcs that are visible from undiscovered (i.e.~good) vertices may be discovered, but these are few as we shall see.
For each undiscovered vertex $v$ and each bin $B_i$, let $u^+_i(v)$ be the number of undiscovered out-arcs visible from $v$ and whose other end is in $B_i$, and let $U^+_i(v)$ be the set of unexposed endvertices of these out-arcs.
Similarly, let $u^-_i(v)$ be the number of undiscovered in-arcs visible from $v$ and whose other end is in $B_i$, and let $U^-_i(v)$ be the set of unexposed endvertices of these in-arcs.
Property \textbf{H1} of $\cH$ guarantees the following:
\begin{enumerate}[{\bf U1:}]
\item
For any bins $B_j,B_i$ (possibly with $i=j$) and any undiscovered vertex $v\in B_j$,
\[
\displaystyle u^+_i(v) \ge d^+_i(v) -4k \geq 2,\qquad u^-_i(v) \ge d^-_i(v) -4k\ge 2.
\]
\end{enumerate}
Moreover, since the unexposed endpoint of each undiscovered arc remains random, we get:
\begin{enumerate}[{\bf U1:}]
\addtocounter{enumi}{1}
\item
$U^+_i(v)$ and $U^-_i(v)$ are two subsets of $\hat B_i\setminus\{v\}$ of sizes $u^+_i(v)$ and $u^-_i(v)$ chosen independently and u.a.r.\ (and also independently of the arcs that are visible from other vertices).
\end{enumerate}
Note that {\bf U1}--{\bf U2} are desired properties in light of Lemma~\ref{mainlemma}. We shall see that they remain valid after we expose additional arcs and slightly modify our digraph.

Unfortunately, the arcs in $\cD_{n,p_-}$ do not suffice to build a $\pi$-HC, since the necessary conditions of event $\cA$ do not hold at time $p_-$ (in view of our assumption that $p_-<p_*$).
To that effect, we expose all the arcs in $\cD_{n,p_*}$ that are incident with a dangerous vertex and label them discovered, as well. Since $p_*$ is a random time, here is a more careful description of this operation. Let us first expose the number $r$ of arcs (but not their locations) in $\cD_{n,p_+} \setminus \cD_{n,p_-}$, and list them as $e_1,\ldots,e_r$ in the order they appear in the process.
For each $1\le i\le r$, we determine whether or not $e_i$ is incident with a dangerous vertex, and if it is then we fully disclose the two endvertices of $e_i$ and label $e_i$ discovered. Otherwise, the endpoints of $e_i$ remain unknown (such $e_i$ will not play any role in our construction of the $\pi$-HC). Let $r_*$ be the smallest $i\in\{0,1,\ldots,r\}$ such that $\cD_{n,p_-} \cup\{e_1,\ldots,e_{i}\}$ satisfies the conditions in event $\cA$. From our assumption that $p_- < p_* < p_+$, we have $1\le r_*\le r-1$ and $\cD_{n,p_*} = \cD_{n,p_-} \cup\{e_1,\ldots,e_{r_*}\}$.
Let $\cD_* = \cD_{n,p_-} \cup\{e_i : \text{$e_i$ is discovered, } 1\le i\le r_*\}$. By construction $\cD_*$ satisfies $\cA$ and $\cH$ (since $\cD_* \subseteq \cD_{n,p_*} \subseteq \cD_{n,p_+}$ and $\cH$ is monotone decreasing).
Moreover, the additional newly discovered arcs in $\{e_1,\ldots,e_{r_*}\}$ are disjoint with the undiscovered arcs of $\cD_{n,p_-}$, since the latter are only incident with good vertices. Hence, the joint distribution of the sets $U^\pm_i(v)$ of unexposed endpoints of undiscovered arcs is unaffected by the additional information revealed, and in particular properties {\bf U1}--{\bf U2} are still valid.
Our goal is to build a $\pi$-HC in $\cD_*$ (or just a $(\ra,\la)$-HC when $\pi=(\ra,\la,\ra,\la)$ and $n\equiv 2\mod 4$). In the sequel, we will expose some additional information about $\cD_*$ and make slight modifications to it (such as moving some vertices from one bin to another).
By abusing slightly the notation, we will still use $\cD_*$ to denote the probability distribution of the modified random digraph given all the exposed information.

\paragraph{Step 2 -- short path collection:}
Our next step is to build a pairwise vertex-disjoint collection $\sP$ of short paths, each of which contains exactly one vertex that is not good.
To make this more precise, we first consider the case where $k\mid n$ (i.e.~either $\pi$ is non-alternating or $\pi=(\ra,\la,\ra,\la)$ with $4\mid n$). For each non-good vertex $w$ (that is, $w$ is~bad or dangerous), we will build a path $P(w)=v_1v_2\cdots v_{6k+1}$ of length $6k$ with the following properties:
$w$ is an internal vertex of $P(w)$;
all the vertices in $P(w)$ are good except for $w$ (so in particular $w\mapsto P(w)$ is a bijection between non-good vertices and paths in $\sP$);
each arc $v_iv_{i+1}$ is oriented as $\pi_i$ for $i=1,\dots, 6k$;
and each $v_i\in B_i$ for $i=1,\dots, 6k+1$. (Recall that modular arithmetic is used for both bins and pattern positions so, for instance, $v_{6k+1} \in B_1$.)
Note that in order to achieve this last property, we may have to move some vertices from one bin to another, but this will be done while keeping the bin sizes unchanged.
Moreover, only internal vertices of paths in $\sP$ may be moved to a different bin.
When $\pi=(\ra,\la,\ra,\la)$ with $n\equiv 2 \mod 4$, we define the collection of paths $\sP$ as above with the only proviso that exactly one of the paths in $\sP$ must have length $6k+2$ instead of $6k$.
In this longer path $P=v_1v_2\cdots v_{6k+3}$, each vertex $v_i\in B_i$ for $i=1,\dots, 6k+2$ but $B_{6k+3}\in B_1$ (so the path cycles $6$ times through bins $B_1,B_2,B_3,B_4$, and then visits $B_1$, $B_2$ and again $B_1$), and arcs are oriented in an alternating fashion with $v_iv_{i+1}$ oriented as $\pi_i$ for $i=1,\dots, 6k+2$.

We will show that we can build a such a collection of paths $\sP$.
Assume inductively that we have already created a sub-collection $\sP_0\subset \sP$ of such paths (possibly $\sP_0=\emptyset$ is the trivial sub-collection) and let $w_2$ be a non-good vertex not contained in any path in $\sP_0$.
Our goal is to build a new path $P(w_2)$ with all the desired properties and add it to $\sP_0$.
We will first suppose that $k\mid n$ (that includes the case when $\pi=(\ra,\la,\ra,\la)$ with $4\mid n$), and then show how to modify the construction if $\pi=(\ra,\la,\ra,\la)$ and $n\equiv 2 \mod 4$.

\paragraph{Step 2a -- starting path $\mbf{P(w_2)}$:}
Recall that $N(w_2)=N_{\cD_*}(w_2)$ is the set of neighbours of $w_2$ (that is, those vertices that are joined to $w_2$ by arcs in $\cD_*$ with any label and orientation). This set has already been exposed since $w_2$ is not good.
We claim that we can pick two different neighbours $w_1,w_3 \in N(w_2)$ that are good and are not contained in any path in $\sP_0$. 
Moreover, for the alternating pattern, we can pick $w_1,w_3$ so that arcs $w_1w_2$ and $w_3w_2$ are either both oriented out of $w_2$ or both into $w_2$.
Indeed, if $w_2$ is bad, then $|N(w_2)| \ge |N_{\cD_{n,p_-}}(w_2)| \ge 4k+3$ by the  definition of being bad. Then $N(w_2)$ contains at least three good vertices not contained in any path in $\sP_0$ by property \textbf{H5} of Lemma~\ref{lem:useful}. By the pigeonhole principle, we can pick two such vertices $w_1$ and $w_3$ that are joined to $w_2$ by arcs with the same orientation. On the other hand, if $w_2$ is dangerous, all its neighbours are good and do not intersect any path in $\sP_0$ by property \textbf{H4}. In that case, event $\cA$ guarantees the existence of appropriate vertices $w_1,w_3$ (with arcs $w_1w_2$ and $w_3w_2$ both oriented out of $w_2$ or both into $w_2$, in the alternating case).
Note that this is the only point in the construction of path $P(w_2)$ that we may use arcs in $\cD_*$ that are not in $\cD_{n,p_-}$.

\paragraph{Step 2b -- growing path $\mbf{P(w_2)}$:}
We first show that we can pick a position $j\in[k]$ in the pattern such that $w_1w_2$ and $w_2w_3$ are oriented as $\pi_{j}$ and $\pi_{j+1}$.
First suppose that arcs $w_1w_2$ and $w_3w_2$ are both oriented out of $w_2$ or both into $w_2$ (note that these are the only two possible cases in the alternating case, from our choice of $w_1$ and $w_3$).
Clearly for $\pi=(\ra,\la,\ra,\la)$ and also for any primitive pattern $\pi$ of length $k\ge3$,
there must be some $j\in[k]$ such that $\pi_{j}=\la$ and $\pi_{j+1}=\ra$ (and similarly some $j'\in [k]$ such that $\pi_{j'}=\ra$ and $\pi_{j'+1}=\la$). In either case, we can pick a position $j$ with $w_1w_2$ and $w_2w_3$ oriented as $\pi_j$ and $\pi_{j+1}$, as claimed.
For $k\ge 3$, there is an additional case to be considered. Suppose that exactly one of the arcs $w_1w_2$ and $w_3w_2$ is oriented out of $w_2$ and the other one is into $w_2$. Since $\pi$ is not the alternating pattern, there must be some $j\in[k]$ such that $\pi_{j}=\pi_{j+1}=\ra$ or $\pi_{j}=\pi_{j+1}=\la$. Hence (switching vertices $w_1$ and $w_3$ if needed), we can conclude that $w_1w_2$ and $w_2w_3$ are oriented as $\pi_{j}$ and $\pi_{j+1}$, also for this case.
Let $B_{a_1},B_{a_2},B_{a_3}$  (possibly $a_1=a_2=a_3$) be the bins containing vertices $w_1,w_2,w_3$, respectively.
We will find a path $P(w_2)=w_{-k-j+2}\cdots w_0w_1w_2w_3\cdots w_{5k-j+2}$ of length $6k$ in $\cD_*$ with the following properties: $P(w_2)$ does not intersect any path in $\sP_0$; all vertices $w_{-k-j+2},\ldots, w_{5k-j+2}$ (of course, except for $w_2$) are good; each arc $w_iw_{i+1}$ is oriented as $\pi_{i+j-1}$ for $i=-k-j+2,\dots, 5k-j+1$; and each $w_i\in B_{i+j-1}$ for $i=-k-j+2,\dots, 5k-j+2$, except for
\begin{equation}\label{eq:exceptional}
\begin{cases}
w_1\in B_{a_1}\\
w_2\in B_{a_2}\\
w_3\in B_{a_3}
\end{cases}
\qquad\text{and}\qquad
\begin{cases}
w_{a_1+2k-j+1}\in B_{j}\\
w_{a_2+3k-j+1}\in B_{j+1}\\
w_{a_3+4k-j+1}\in B_{j+2}.
\end{cases}
\end{equation}
The bins of vertices $w_i$ and $w_{a_i+(i+1)k-j+1}$ will be swapped for $i=1,2,3$ (unless $a_i=i+j-1$, in which case both $w_i$ and $w_{a_i+(i+1)k-j+1}$ will be placed correctly in the same bin $B_{a_i} = B_{i+j-1}$). Note that since $j,a_1,a_2,a_3\in[k]$ and $k\ge 3$,
\[
-k-j+2 < 1 < 2 < 3 < a_1+2k-j+1 < a_2+3k-j+1 < a_3+4k-j+1 < 5k-j+2,
\]
so in particular all six vertices $w_1,w_2,w_3,w_{a_1+2k-j+1},w_{a_2+3k-j+1},w_{a_3+4k-j+1}$ must be distinct and internal vertices of the path.
For each $i=-k-j+2,\dots, 5k-j+2$, let $B_{\a_i}$ be the bin where we wish to find vertex $w_i$. Then $\a_i=i+j-1$, with the possible exceptions described in~\eqref{eq:exceptional}.
We will grow our path $P(w_2)$ by adding vertices to each end of $w_1w_2w_3$, one vertex at a time.
We can achieve this as follows.
For each $i=3,\ldots,5k-j+1$, expose all in- and out-arcs of $\cD_*$ incident with $w_i$ (visible and invisible from $w_i$). We label all these exposed arcs as discovered and vertex $w_i$ is labelled discovered as well.
Since $w_i$ is good, it has at least $4k+2$ neighbours $w$ in bin $B_{\a_{i+1}}$ with $w_iw$ oriented as $\pi_{i+j-1}$. By property {\bf H5}, at least $2$ of these vertices are good and not in any path $P\in\sP_0\cup\{w_1w_2w_3\cdots w_i\}$. Pick one of them and call it $w_{i+1}$. By an analogous argument, we also grow the path from the other end, by adding a suitable vertex $w_{i-1}$ for each $i=1,0,-1,\ldots,-k-j+3$, and thus complete our path $P(w_2)$. This path has all the required properties, except for the fact that vertices listed in~\eqref{eq:exceptional} may be placed in the wrong bin (depending on the values of $a_1,a_2,a_3$). Hence, we swap the bins of the vertices in the path that were misplaced: that is, $w_i$ is moved to bin $B_{i+j-1}$ and $w_{a_i+(i+1)k-j+1}$ is moved to $B_{a_i}$, for $i=1,2,3$. Note that the sizes of the bins remain unchanged and thus balanced. Moreover, writing $P(w_2)=v_1v_2\cdots v_{6k+1}$ with $v_i = w_{i-k-j+1}$, we have that every vertex $v_i$ ($i=1,\ldots,6k+1$) belongs to bin $B_i$, and each arc $v_iv_{i+1}$ ($i=1,\ldots,6k$) is oriented as $\pi_i$, as required. Then, we can extend $\sP_0$ to $\sP_0\cup\{P(w_2)\}$ and inductively obtain our desired collection of paths $\sP$.

When $\pi=(\ra,\la,\ra,\la)$ and $n\equiv 2 \mod 4$, we just need to extend one of the paths $P=v_1v_2\cdots v_{6k+1}$ in $\sP$ two more steps.
When doing so, we must make sure that the two additional vertices $v_{6k+2},v_{6k+3}$ belong to the right bins and the edges $v_{6k+1}v_{6k+2}, v_{6k+2}v_{6k+3}$ have the appropriate orientations. This can be achieved by the exact same argument as above, and thus we omit the details.
Note that we have tacitly assumed that $\sP$ contains at least one path, which would not be true if all the vertices of $\cD_*$ were good\footnote{It is easy to show that a.a.s.~this does not happen, but we do not make use of this fact.}. In that case, we could still create one path by applying Step~2a to a good vertex $w_2$, and extending this path as in Step~2b.

\paragraph{Step 3 -- path contraction:}
A crucial property in the above construction of $\sP$ is that we have only exposed (and labelled discovered) those arcs that are incident with discovered vertices, which are precisely the internal vertices of the paths in $\sP$. All the other vertices remain undiscovered (and are good). Note that some arcs incident with undiscovered edges may be discovered. However, in view of {\bf H5} (applying Lemma~\ref{lem:useful} with $\sP_0=\sP$), we infer that each undiscovered vertex $v$ is incident with at most $4k$ discovered arcs.
Hence, property {\bf U1} remains valid for any undiscovered vertex $v\in B_j$ and bins $B_i,B_j$,  with the updated values of $\hat B_i$, $U^\pm_i(v)$ and $u^\pm_i(v)$.
Moreover, since we did not reveal any information about undiscovered arcs, the unexposed endpoints of the undiscovered arcs visible from $v$ remain random, and therefore property {\bf U2} also holds.
Now we contract each path in $\sP$ into a vertex. Vertices obtained from a contracted path are called {\em fat} and are placed in bin $B_1$, while other vertices are called {\em ordinary}. We declare fat vertices to be undiscovered, so all vertices in the contracted graph are undiscovered.
Let $n'$ be the number of vertices (ordinary or fat) that remain in the digraph after the path contractions. Clearly $n/(7k)\le n'\le n$, since each path has fewer than $7k$ vertices. In particular, $n'\to\infty$.
We claim that $k\mid n'$ and moreover, after contracting the paths, every bin contains exactly $n'/k$ vertices.
Indeed, when $k\mid n$, each path in $\sP$ has $6$ vertices in each bin $B_i$ ($i=2,\ldots, k$) and $7$ vertices in bin $B_1$, so path contractions maintain the bin sizes balanced. 
On the other hand, when $\pi=(\ra,\la,\ra,\la)$ and $n\equiv 2\mod 4$, two additional vertices---one in bin $B_1$ and one in $B_2$---are contracted (as a result of the longer path), so bins $B_1,B_2,B_3,B_4$ become balanced as well in this case.

\paragraph{Step 4 -- end of the proof:}
At this final stage, we will only need arcs between vertices in consecutive bins $B_i$ and $B_{i+1}$, for each $i\in[k]$, so we delete all the remaining arcs. Moreover, for each fat vertex $v$ obtained by contracting a path $P=v_1\cdots v_{6k+1}$, we delete all arcs (in- and out-arcs, visible and invisible) that are incident with $v_1$ except for those whose other end is in bin $B_k$, which remain unexposed. Similarly, we delete all arcs that are incident with $v_{6k+1}$ except for those whose other end is in bin $B_2$.
In other words, fat vertex $v$ plays the role of vertex $v_1$ (respectively, $v_{6k+1}$) when concerned with arcs between bins $B_1$ and $B_k$ (respectively, $B_1$ and $B_2$).
Let us call the resulting digraph $\cD_{**}$.
(Recall that $\cD_{**}$ has $n'\to\infty$ vertices with $k\mid n'$, and each bin contains exactly $n'/k$ vertices.)
We claim that, after all these edge deletions, properties {\bf U1}--{\bf U2} remain valid for any vertex $v\in B_j$ and bin $B_i$ with $i\in\{j-1,j+1\}$.\footnote{The proof of this claim makes use of the fact that $k\ge3$ and hence arcs between bins $B_1$ and $B_2$ are different from those between $B_1$ and $B_k$. This justifies our choice to use the pattern $(\ra,\la,\ra,\la)$ in the analysis of the alternating case. Other approaches are also possible, but we believe this is the simplest one.} This is clearly true if $v$ is a fat vertex corresponding to a path $P=v_1\cdots v_{6k+1}$, since $v$ inherits the arcs visible from $v_1$ with opposite end in $B_k$ and the arcs visible from $v_{6k+1}$ with opposite end in $B_2$. On the other hand, for any ordinary vertex $v\in B_j$ and bin $B_i$ ($i\in\{j-1,j+1\}$), the only arcs visible from $v$ that may have been deleted are those incident with a fat vertex (and thus with some vertex of a path in $\sP$). Hence, by property {\bf H5} as before, we conclude that {\bf U1} remains valid. Property {\bf U2} also follows from the fact that we have not exposed any additional information about the surviving arcs and thus they remain uniformly distributed.
Finally, {\bf U1} and {\bf U2} imply that the contracted digraph $\cD_{**}$ contains a copy of $D_{S\text{-in},T\text{-out}}$ (on $n'$ vertices with $n'\to\infty$) with $s_{i,i+1},s_{i,i-1},t_{i,i-1},t_{i,i-1} \ge 2$. So we can apply Lemma~\ref{mainlemma} and find a $\pi$-HC in $\cD_{**}$ a.a.s.
Finally, when $\pi$ is primitive of length $k\ge3$, un-contracting the paths yields the desired $\pi$-HC in $\cD_* \subseteq \cD_{n,p_*}$. When $\pi=(\ra,\la,\ra,\la)$, we simply get a $(\ra,\la)$-HC.
This finishes the proof of Theorem~\ref{thm1}.
Corollaries~\ref{cor1} and~\ref{cor2} follow immediately as a consequence of Lemma~\ref{lem:PH}.

\section{Final Comments}
We have proved a hitting time result for the existence of a Hamilton cycle with edges oriented according to a periodic pattern. We have shown that in some cases this typically requires asymptotically one half of that needed for a consistently oriented Hamilton cycle. It would be of some interest to discover how many such cycles there are typically and as to whether there are all such cycles, up to a bound on the length of the pattern $\pi$. As already mentioned, Ferber and Long \cite{FL} have shown that Hamilton cycles with arbitrary orientations occur a.a.s.\ at $m\approx n\log n$. It would be of interest to understand which class of patterns occur a.a.s.\ for $m\leq cn\log n$ where $0<c<1$ is a constant.

\end{document}